\documentclass[11pt,leqno]{amsart}
\usepackage{amssymb,amsmath,amsfonts}
\usepackage[margin=1in]{geometry}
\usepackage{mathrsfs}
\usepackage{color,xcolor}
\usepackage{graphicx}
\usepackage{relsize}
\definecolor{cobalt}{RGB}{61,99,181}
\usepackage[colorlinks, citecolor=cobalt, linkcolor=cobalt]{hyperref}

\allowdisplaybreaks

\newtheorem{thm}{Theorem}[section]

\newtheorem{cor}[thm]{Corollary}
\newtheorem{lem}[thm]{Lemma}

\newtheorem{prop}[thm]{Proposition}

\numberwithin{equation}{section}

\def\D{\mathbb{D}}

\newcommand{\N}{\mathbb{N}}

\newcommand{\C}{\mathbb{C}}

\makeatletter

\newcommand{\Rmnum}[1]{\expandafter\@slowromancap\romannumeral #1@}
\makeatother

\begin{document}
\title[Hyponormal dual Toeplitz operators]{Hyponormal dual Toeplitz operators on the orthogonal complement of the Harmonic Bergman space}

\author[Chongchao Wang]{Chongchao Wang\textsuperscript{1}}
\address{\textsuperscript{1} College of Mathematics and Statistics, Chongqing University, Chongqing, 401331, P. R. China}
\email{chongchaowang@cqu.edu.cn}
\author[Xianfeng Zhao]{Xianfeng Zhao\textsuperscript{2}}
\address{\textsuperscript{2} College of Mathematics and Statistics, Chongqing University, Chongqing, 401331, P. R. China}
\email{xianfengzhao@cqu.edu.cn}

\date{\today}



\begin{abstract}
In this paper, we  mainly  study the  hyponormality of  dual Toeplitz operators on the orthogonal complement of the harmonic Bergman space. First we show that the dual Toeplitz operator with bounded symbol  is hyponormal if and only if it is normal. Then we obtain a necessary and sufficient condition for the dual Toeplitz operator $S_\varphi$ with the symbol $\varphi(z)=az^{n_1}\overline{z}^{m_1}+bz^{n_2}\overline{z}^{m_2}$ ($n_1,n_2,m_1,m_2\in \mathbb N\ \text{and}\ a,b\in\C)$ to be hyponormal. Finally, we show that the rank of the commutator of two dual Toeplitz operators must be an even number if the commutator has a  finite rank.
\end{abstract}

\maketitle

\section{Introduction}

A bounded linear operator $T$ on some  Hilbert space $\mathscr{H}$ is called hyponormal if the commutator  $T^*T-TT^*$ of $T^*$ (the adjoint of $T^*$) and $T$ is positive, i.e.,
$$\langle (T^*T-TT^*)f,f\rangle\geqslant 0$$
for all $f\in \mathscr{H}$. This  definition
was introduced by  Halmos in 1950 and generalizes the concept of a normal operator (where $T^*T = TT^*$). Normal operators are completely understood by people. In fact, it is possible to give a model for an arbitrary normal operator, see \cite[Chapter 2]{Arv} if necessary. The class of hyponormal operators is an important class of non-normal
operators. The problem ``when a hyponormal operator is normal" has been investigated  by many authors. The theory of normal and hyponormal operators is an extensive and highly developed area, which has made important contributions to a number of problems in
functional analysis, operator theory, analytic  function  theory and mathematical physics.

 Normality and hyponormality of Toeplitz operators on analytic or harmonic function spaces have captured people's attention for a long time. Here, for the  basic knowledge about  Toeplitz operators on the Hardy space and on the Bergman space we refer to the books \cite{Dou} and \cite{Zhu}. The normal Toeplitz operator on the Hardy space was characterized by Brown and Halmos \cite{BH} in 1964. For the classical results about  hyponormal Toeplitz operators on the Hardy space, one can consult \cite{CC} for the case of bounded symbols and \cite{FL}, \cite{HKL}, \cite{NT}, \cite{Yu, Zhu1} for the case of  (trigonometric)  polynomial symbols. The hyponormality of Toeplitz operators with some harmonic polynomial symbols on the Bergman space was studied in \cite{CuC}, \cite{HL, H} and \cite{LuS}. Recently, a sufficient condition for the Toeplitz operator with the symbol $|z|^n+c|z|^s\ (n, s \in \mathbb N, c\in \mathbb C)$ to be hyponormal on  certain weighted Bergman space was established in \cite{LS}. Furthermore, hyponormality of block Toeplitz operators on the vector-valued Hardy space was investigated in \cite{CHL}. Concerning Toeplitz operators on the harmonic function space,  Choe and Lee  obtained  a complete characterization for Toeplitz operators with bounded harmonic symbols to be normal on the harmonic Bergman space \cite{ChL}.

The purpose of this paper is to  describe the  normality and hyponormality of dual Toeplitz operators on the orthogonal complement of the harmonic Bergman space. We will  elaborate on this class of operators in the following.

Let $\D$ be the unit disk in the complex plane $\C$ and $dA=\frac{1}{\pi}dxdy$ be the normalized area measure on $\D$.  $L^2(\D,dA)$ denotes the space of the Lebesgue measurable functions on $\D$ with the norm
$$\|f\|=\left(\int_{\D}|f(z)|^2dA(z)\right)^{\frac{1}{2}}<+\infty.$$
Clearly, $L^2(\mathbb D)=L^2(\D,dA)$ is a Hilbert space with the inner product
$$\langle f,g\rangle=\int_{\D}f(z)\overline{g(z)}dA(z).$$

The Bergman space $L_a^2$ is the closed subspace of $L^2(\D)$ consisting of all analytic functions on $\D$.
Similarly, the harmonic Bergman space $L_h^2$ is the closed subspace of $L^2(\D)$ consisting of all harmonic functions on $\D$. Let $P$ and $Q$ be the orthogonal projections from $L^2(\D)$ onto $L_a^2$ and $L_h^2$, respectively. Then $P$ and $Q$ have the following relationship:
\begin{align}\label{Q}
Q(f)=P(f)+\overline{P(\overline{f})}-P(f)(0), \ \ \ \ \ f\in L^2(\mathbb D).
\end{align}

 According to the decomposition  $L^2(\mathbb D, dA)=L_h^2\oplus (L_h^2)^{\bot}$, the multiplication operator $M_\varphi$ with symbol $\varphi\in L^\infty (\mathbb D)$  can be represented as
$$ M_\varphi= \left (\begin{matrix}
   T_\varphi & H_{\overline{\varphi}}^*\\
   H_\varphi & S_\varphi \\
  \end{matrix}\right ),
  $$
where the operator $$S_\varphi u=(I-Q)(\varphi u)$$  is   bounded  and linear  on the Hilbert space $(L_h^2)^{\bot}$. We call $S_\varphi$ the dual Toeplitz operator with
symbol $\varphi$ on  $(L_h^2)^{\bot}$. From the  definition of $S_\varphi$, the following elementary  properties about dual Toeplitz operators are easily verified:\\
$\mathrm{(a)} \ S_\varphi^*=S_{\overline{\varphi}}$;\\
$\mathrm{(b)} \ \|S_\varphi\|\leqslant \|\varphi\|_\infty$;\\
$\mathrm{(c)}\  S_{\alpha \varphi+\beta \psi}=\alpha S_\varphi+\beta S_\psi$ for all $\varphi, \psi\in L^\infty(\mathbb D)$ and all constants $\alpha, \beta\in \mathbb C$.

Unlike the Bergman space and the harmonic Bergman space, the orthonormal basis for the Hilbert  space $(L_h^2)^{\bot}$ can not  be written explicitly. But we are able to study dual Toeplitz operators on this function space
via the functions of the form $z^n\overline{z}^m$ with $m, n \in \mathbb N$. Observe that
$$\mathcal M=\mathrm{span}\Big\{(I-Q)(z^n\overline{z}^m): m, n=1, 2, \cdots\Big\}$$
is a dense subspace of $(L_h^2)^{\bot}$. Using (\ref{Q}), then elementary calculations give us that
\begin{align}\label{I-Q}
\begin{split}
(I-Q)(z^n\overline{z}^m)=\begin{cases}
\mathlarger{z^n\overline{z}^m-\frac{m-n+1}{m+1}\overline{z}^{m-n},~~~~~~~~m>n,}\vspace{2.5mm}\\
\mathlarger{z^n\overline{z}^m-\frac{n-m+1}{n+1}z^{n-m},~~~~~~~~m \leqslant n.}
\end{cases}
\end{split}
\end{align}
Moreover,  it is easy to check that
\begin{equation}\label{1.a}
\langle z^{n}\overline{z}^{m},z^{k}\overline{z}^{l}\rangle=
\begin{cases}
\mathlarger{\frac{2}{m+n+k+l+2},\ \ \ \ n-m=k-l},\vspace{2mm}\\
0,\ \ \ \ \ \ \   \ \ \ \ \ \ \ \ \ \ \  \ \ \ \ \ \ \ \ \ \ \ \text{otherwise},\vspace{2mm}
\end{cases}
\end{equation}
where $(n,m)$ and $(k,l)$ are  two  pairs  of non-negative integers.

 Indeed, the concept of ``\emph{dual Toeplitz operator}" was first introduced by Stroethoff  and  Zheng \cite{SZ1}. The algebraic and spectral properties of dual Toeplitz operator on the orthogonal complement of the Bergman space space $(L_a^2)^{\bot}$ were characterized in \cite{SZ2}. Ding, Wu and the second author studied the invertibility and spectral properties of dual Toeplitz operators on $(L_a^2)^{\bot}$. Moreover, they obtained a necessary and sufficient condition for dual Toeplitz operators with bounded harmonic symbols to be hyponormal on  $(L_a^2)^{\bot}$, see \cite{DWZ} if necessary.
 In addition, dual Toeplitz operators on other function spaces also have been investigated by many authors in recent years, see \cite{Chen1} and  \cite{Lu1, Lu2, Yang1, Yu1} for more details.

However,  little is known about dual Toeplitz operators on the complement of the harmonic Bergman space $(L_h^2)^{\bot}$. In 2015, Yang and Lu completely solved the commuting problem of two dual Toeplitz operators on $(L_h^2)^{\bot}$ with bounded harmonic symbols \cite{YaL}. Therefore, they characterized the normal dual Toeplitz operator on $(L_h^2)^{\perp}$ with  a bounded harmonic symbol. Recently, Peng and the second author \cite{PeZ} characterized the boundedness, compactness, spectral structure and  algebraic  properties of dual Toeplitz operators on  $(L_h^2)^{\perp}$.  As far as we know,  the researches on the normality and hyponormality of dual Toeplitz operators with  non-harmonic symbols on $(L_h^2)^{\bot}$  have  not been reported.

In this paper, we first show that the dual Toeplitz operator on $(L_h^2)^{\bot}$ with bounded symbol is hyponormal if and only if it is normal, see Theorem \ref{2.b}. Moreover, we obtain a necessary and sufficient condition for the dual Toeplitz operator $S_\varphi$ with non-harmonic symbol of the form $\varphi(z)=az^{n_1}\overline{z}^{m_1}+bz^{n_2}\overline{z}^{m_2}$ to be hyponormal on  $(L_h^2)^{\bot}$, where $n_1,n_2,m_1,m_2$ are nonnegative integers and $a, b$ are complex constants, see Theorem \ref{M2} for the details. Moreover, we consider the rank of the commutator of two dual Toeplitz operators with general bounded symbols in Section \ref{4}. We show  that the rank of the commutator of two dual Toeplitz operators must be an even number if the commutator has a finite rank, see Theorem \ref{M3} in the last section.

\section{Preliminaries}
In this section, we will show that the  normality  and  hyponormality  for dual Toeplitz operator on $(L_h^2)^\perp$ are equivalent.
To do so, we need the following well-known result about self-adjoint operators (see \cite{Con} if needed).
\begin{lem}\label{2.a}
Suppose that $T$ is a self-adjoint operator on some  Hilbert space $\mathscr{H}$ and
$$\langle Th,h\rangle=0$$
for all $h\in \mathscr{H}$.
Then
$$\langle Tf,g\rangle=0$$
for any $f,g\in \mathscr{H}$, i.e., $T=0$.
\end{lem}

The first main result of this paper is contained in the following theorem.
\begin{thm}\label{2.b}
Suppose that $\varphi\in L^\infty(\mathbb D)$. Then the dual Toeplitz operator $S_{\varphi}$ is hyponormal on $(L_h^2)^{\perp}$ if and only if it is normal.
\end{thm}
\begin{proof}
Clearly, $S_{\varphi}$ is hyponormal if it is normal. Now we assume that $S_{\varphi}$ is hyponormal, i.e.,
$$ S_\varphi^* S_\varphi-S_\varphi S_\varphi^*\geqslant 0.$$
From the definition of the projection $Q$, we have
\begin{align*}
Q(\varphi f) &= P(\varphi f)+\overline{P(\overline{\varphi f})}-P(\varphi f)(0)\\
             &= \overline{\overline{P(\varphi f)}+P(\overline{\varphi f})-\overline{P(\varphi f)(0)}}\\
             &= \overline{Q(\overline{\varphi f})}
\end{align*}
and
$$Q(\overline{\varphi}f)=\overline{Q(\varphi\overline{f})}$$
for any $f\in (L_{h}^2)^\perp$.

This gives that
$$\left\|S_{\varphi}f\right\|=\left\|(I-Q)(\varphi f)\right\|=\left\|\overline{(I-Q)(\overline{\varphi f})}\right\|=\left\|S_{\varphi}^*\overline{f}\right\|$$
and
$$\left\|S_{\varphi}^*f\right\|=\left\|(I-Q)(\overline{\varphi}f)\right\|=\left\|\overline{(I-Q)(\varphi\overline{f})}\right\|=\left\|S_{\varphi}\overline{f}\right\|$$
for each $f\in (L_{h}^2)^\perp$. Thus we have
\begin{align*}
0 &\leqslant \big\langle (S_{\varphi}^*S_{\varphi}-S_{\varphi}S_{\varphi}^*)f,f\big\rangle\\
  &= \left\|S_{\varphi}f\right\|^2-\left\|S_{\varphi}^*f\right\|^2\\
  &= \left\|S_{\varphi}^*\overline{f}\right\|^2-\left\|S_{\varphi}\overline{f}\right\|^2\\
  &= -\langle (S_{\varphi}^*S_{\varphi}-S_{\varphi}S_{\varphi}^*)\overline{f},\overline{f}\rangle\\
  &\leqslant 0,
\end{align*}
where the last inequality follows from that $\overline{f}$ is also in $(L_h^2)^{\bot}$.
This implies that $$\langle (S_{\varphi}^*S_{\varphi}-S_{\varphi}S_{\varphi}^*)f,f\rangle=0$$ for all $f\in (L_{h}^2)^\perp$.
By Lemma \ref{2.a}, we have $S_{\varphi}^*S_{\varphi}-S_{\varphi}S_{\varphi}^*=0$, which means that $S_{\varphi}$ is normal. This completes the proof of Theorem \ref{2.b}.
\end{proof}

Combining the above theorem with \cite[Corollary 2.1]{YaL}, we obtain the following characterization for the hyponormal dual Toeplitz operators with bounded harmonic symbols on $(L_h^2)^{\bot}$.
\begin{cor}\label{2.e}
Suppose that $\varphi\in L^\infty(\D)$ is a harmonic function. Then $S_{\varphi}$ is hyponormal if and only if
$a\varphi+b\overline{\varphi}$ is constant on $\mathbb D$, where $a$ and $b$ are complex  constants, not both zero.
\end{cor}

We end this section by discussing the hyponormality of the special dual Toeplitz operator $S_{z^n\overline{z}^m}$ (where $m, n\in \N$) in the following proposition, which is useful for us to prove our main theorem in the next section.
\begin{prop}\label{2.c}
Let $n,m$ be two nonnegative integers, then $S_{z^n\overline{z}^m}$ is hyponormal (normal) on $(L_h^2)^{\perp}$ if and only if $n=m$.
\end{prop}
\begin{proof}
If $n=m$, then $S_{z^n\overline{z}^m}$ is self-adjoint, and hence is hyponormal. Next, we will show  $m=n$  provided that $S_{z^n\overline{z}^m}$ is hyponormal. In view of Theorem \ref{2.b}, we suppose that $S_{z^n\overline{z}^m}$ is normal.

Let $k$ be an arbitrary positive  integer such that $k>\max\{m, n\}$. By (\ref{I-Q}) in Section 1, we have
$$f_k(z):=z^k\overline{z}-Q(z^k\overline{z})=z^k\overline{z}-\frac{k}{k+1}z^{k-1}$$
and $f_k\in (L_h^2)^{\perp}$. Thus we obtain
\begin{equation}\label{eq3}
\begin{aligned}
&S_{z^n\overline{z}^m}(f_k)\\
&=S_{z^n\overline{z}^m}\Big(z^k\overline{z}-\frac{k}{k+1}z^{k-1}\Big)\\
&=(I-Q)(z^{n+k}\overline{z}^{m+1})-\frac{k}{k+1}(I-Q)(z^{n+k-1}\overline{z}^m)\\
&=z^{n+k}\overline{z}^{m+1}-\frac{n+k-m}{n+k+1}z^{n+k-m-1}-\frac{k}{k+1}z^{n+k-1}\overline{z}^m+\frac{k(n+k-m)}{(k+1)(n+k)}z^{n+k-m-1}\\
&=z^{n+k}\overline{z}^{m+1}-\frac{k}{k+1}z^{n+k-1}\overline{z}^m-\frac{n(n+k-m)}{(k+1)(n+k)(n+k+1)}z^{n+k-m-1}.
\end{aligned}
\end{equation}
Similarly, we have
\begin{equation}\label{eq4}
\begin{aligned}
S_{z^n\overline{z}^m}^{*}(f_k)
&=S_{z^m\overline{z}^n}(f_k)\\
&=z^{m+k}\overline{z}^{n+1}-\frac{k}{k+1}z^{m+k-1}\overline{z}^n-\frac{m(m+k-n)}{(k+1)(m+k)(m+k+1)}z^{m+k-n-1}.
\end{aligned}
\end{equation}
Before computing $\|S_{z^n\overline{z}^m}(f_k)\|$ and $\|S_{z^n\overline{z}^m}^{*}(f_k)\|$, some elementary calculations are required. Direct computations give us the following equalities:
$$\|z^{n+k}\overline{z}^{m+1}\|^2=\|z^{m+k}\overline{z}^{n+1}\|^2=\frac{1}{n+k+m+2},$$
$$\|z^{n+k-1}\overline{z}^m\|^2=\|z^{m+k-1}\overline{z}^n\|^2=\frac{1}{n+k+m},$$
$$\|z^{n+k-m-1}\|^2=\frac{1}{n+k-m}, \ \ \ \ \ \|z^{m+k-n-1}\|^2=\frac{1}{m+k-n},$$
$$\left\langle z^{n+k}\overline{z}^{m+1},z^{n+k-1}\overline{z}^m\right\rangle=\left\langle z^{m+k}\overline{z}^{n+1},z^{m+k-1}\overline{z}^n\right\rangle=\frac{1}{n+k+m+1},$$
$$\left\langle z^{n+k-1}\overline{z}^m,z^{n+k-m-1}\right\rangle=\frac{1}{n+k},\ \ \ \left\langle z^{m+k-1}\overline{z}^n,z^{m+k-n-1}\right\rangle=\frac{1}{m+k},$$
$$\left\langle z^{m+k}\overline{z}^{n+1},z^{m+k-n-1}\right\rangle=\frac{1}{m+k+1}, \ \ \ \left\langle z^{n+k}\overline{z}^{m+1},z^{n+k-m-1}\right\rangle=\frac{1}{n+k+1}.$$

Combining the above equalities, we obtain that
\begin{align*}
\|S_{z^n\overline{z}^m}(f_k)\|^2&=\langle S_{z^n\overline{z}^m}(f_k), S_{z^n\overline{z}^m}(f_k)\rangle\\
&=\frac{1}{n+k+m+2}-\frac{2k}{(k+1)(n+k+m+1)}\\
&\quad-\frac{2n(n+k-m)}{(k+1)(n+k)(n+k+1)^2}+\frac{k^2}{(k+1)^2(n+k+m)}\\
&\quad+\frac{2kn(n+k-m)}{(k+1)^2(n+k)^2(n+k+1)}+\frac{n^2(n+k-m)}{(k+1)^2(n+k)^2(n+k+1)^2}
\end{align*}
and
\begin{align*}
\|S_{z^n\overline{z}^m}^*(f_k)\|^2
&=\frac{1}{m+k+n+2}-\frac{2k}{(k+1)(m+k+n+1)}\\
&\quad-\frac{2m(m+k-n)}{(k+1)(m+k)(m+k+1)^2}+\frac{k^2}{(k+1)^2(m+k+n)}\\
&\quad+\frac{2km(m+k-n)}{(k+1)^2(m+k)^2(m+k+1)}+\frac{m^2(m+k-n)}{(k+1)^2(m+k)^2(m+k+1)^2}.
\end{align*}
Since $S_{z^n\overline{z}^m}$ is normal, we have
\begin{align*}
0&=\langle (S_{z^n\overline{z}^m}^*S_{z^n\overline{z}^m}-S_{z^n\overline{z}^m}S_{z^n\overline{z}^m}^*)f_k, f_k\rangle=\|S_{z^n\overline{z}^m}(f_k)\|^2-\|S_{z^n\overline{z}^m}^*(f_k)\|^2\\
&=\frac{2kn(n+k-m)}{(k+1)^2(n+k)^2(n+k+1)}+\frac{n^2(n+k-m)}{(k+1)^2(n+k)^2(n+k+1)^2}\\
&\quad -\frac{2n(n+k-m)}{(k+1)(n+k)(n+k+1)^2}-\frac{2km(m+k-n)}{(k+1)^2(m+k)^2(m+k+1)}\\
&\quad -\frac{m^2(m+k-n)}{(k+1)^2(m+k)^2(m+k+1)^2}+\frac{2m(m+k-n)}{(k+1)(m+k)(m+k+1)^2}\\
&=\frac{-n^2(n+k-m)}{(k+1)^2(n+k)^2(n+k+1)^2}+\frac{m^2(m+k-n)}{(k+1)^2(m+k)^2(m+k+1)^2}.
\end{align*}
It follows that
\begin{equation}\label{eq1}
\frac{n^2(n+k-m)}{(k+1)^2(n+k)^2(n+k+1)^2}=\frac{m^2(m+k-n)}{(k+1)^2(m+k)^2(m+k+1)^2}
\end{equation}
for each  $k>\max\{m, n\}$.
Multiplying both sides of (\ref{eq1}) by
$$(k+1)^2(n+k)^2(n+k+1)^2(m+k)^2(m+k+1)^2$$
gives us
\begin{equation}\label{eq2}
n^2(n+k-m)(m+k)^2(m+k+1)^2=m^2(m+k-n)(n+k)^2(n+k+1)^2
\end{equation}
for every $k>\max\{m, n\}$.

Let $$p(z)=n^2(n+z-m)(m+z)^2(m+z+1)^2-m^2(m+z-n)(n+z)^2(n+z+1)^2.$$
Then $p$ is a polynomial and  $p$ has finitely many zeros in the complex plane. But (\ref{eq2}) implies that $p$ must be a zero polynomial. This yields that the coefficients of $z^5$ is zero, that is,
$$n^2-m^2=0.$$
This implies that $n=m$, hence the proof of Proposition \ref{2.c} is finished.
\end{proof}

\section{dual Toeplitz operator with the symbol $\varphi(z)=az^{n_1}\overline{z}^{m_1}+bz^{n_2}\overline{z}^{m_2}$}

This section is devoted to establishing a necessary and sufficient condition for dual Toeplitz operator with the symbol $\varphi(z)=az^{n_1}\overline{z}^{m_1}+bz^{n_2}\overline{z}^{m_2}$ to be hyponormal (normal) on the orthogonal complement of the harmonic Bergman space. In view of Proposition \ref{2.c}, we need only to consider the case of $ab\neq 0$ for the symbol $\varphi$. The main result of this section is the following theorem.
\begin{thm}\label{M2}
Let $\varphi(z)$ be the nonzero function $az^{n_1}\overline{z}^{m_1}+bz^{n_2}\overline{z}^{m_2}$, where $a, b$ are nonzero constants  and $n_1,n_2,m_1,m_2$ are positive integers. Then the dual Toeplitz operator $S_{\varphi}$  is hyponormal (normal) on $(L_h^2)^{\perp}$ if and only if one of the following conditions holds:\\
$\mathrm{(1)}$ $n_1=m_1=n_2=m_2$;\\
$\mathrm{(2)}$ $n_1=m_1,n_2=m_2,n_1\neq n_2$ and $\arg(a)=\arg(b)$;\\
$\mathrm{(3)}$ $n_1=m_2,m_1=n_2,n_1\neq m_1$ and $|a|=|b|$.
\end{thm}

As the proof of Theorem \ref{M2} is complicated, we shall divide its proof into several lemmas.

\begin{lem}\label{2.d}
Let $\varphi(z)$  be the nonzero function $z^{n_1}\overline{z}^{m_1}+\alpha z^{n_2}\overline{z}^{m_2}$, where $\alpha$ is a nonzero complex constant and $n_1,n_2,m_1,m_2$ are positive integers. Suppose that
$$(n_1-m_1)(n_2-m_2)\geqslant 0.$$
Then $n_1-m_1=n_2-m_2$ if  $S_{\varphi}$ is normal on $(L_h^2)^{\perp}$.
\end{lem}
\begin{proof}
 Suppose, to the contrary, that $n_1-m_1\neq n_2-m_2$. Using the same argument as in the proof of Proposition \ref{2.c}, we first
compute $\|S_{\varphi}f_k\|^2$ and $\|S_{\varphi}^*f_k\|^2$, where $k$ is an integer such that $k>\max\{n_1,n_2,m_1,m_2\}$ and $f_k$ is defined by
$$f_k(z)=z^k\overline{z}-Q(z^k\overline{z})=\Big(z^k\overline{z}-\frac{k}{k+1}z^{k-1}\Big)\in (L_h^2)^\perp.$$
By (\ref{eq3}) and (\ref{eq4}) in the proof of Proposition \ref{2.c}, we have
$$S_{z^{n_1}\overline{z}^{m_1}}(f_k)=z^{n_1+k}\overline{z}^{m_1+1}-\frac{k}{k+1}z^{n_1+k-1}\overline{z}^{m_1}
-\frac{n_1(n_1+k-m_1)}{(k+1)(n_1+k)(n_1+k+1)}z^{n_1+k-m_1-1},$$
$$S_{z^{n_2}\overline{z}^{m_2}}(f_k)=z^{n_2+k}\overline{z}^{m_2+1}-\frac{k}{k+1}z^{n_2+k-1}\overline{z}^{m_2}
-\frac{n_2(n_2+k-m_2)}{(k+1)(n_2+k)(n_2+k+1)}z^{n_2+k-m_2-1},$$
$$S_{z^{n_1}\overline{z}^{m_1}}^{*}(f_k)=z^{m_1+k}\overline{z}^{n_1+1}-\frac{k}{k+1}z^{m_1+k-1}\overline{z}^{n_1}
-\frac{m_1(m_1+k-n_1)}{(k+1)(m_1+k)(m_1+k+1)}z^{m_1+k-n_1-1}$$
and
$$S_{z^{n_2}\overline{z}^{m_2}}^{*}(f_k)=z^{m_2+k}\overline{z}^{n_2+1}-\frac{k}{k+1}z^{m_2+k-1}\overline{z}^{n_2}
-\frac{m_2(m_2+k-n_2)}{(k+1)(m_2+k)(m_2+k+1)}z^{m_2+k-n_2-1}.$$

Since $n_1-m_1\neq n_2-m_2$, we observe by  (\ref{1.a})  that
$$S_{z^{n_1}\overline{z}^{m_1}}(f_k)\perp S_{z^{n_2}\overline{z}^{m_2}}(f_k)\ \ \ \
\mathrm{and}\ \ \ \ S_{z^{n_1}\overline{z}^{m_1}}^{*}(f_k)\perp S_{z^{n_2}\overline{z}^{m_2}}^{*}(f_k).$$
This leads to
\begin{align*}
\|S_{\varphi}f_k\|^2&=\|S_{z^{n_1}\overline{z}^{m_1}}(f_k)\|^2+|\alpha|^2\|S_{z^{n_2}\overline{z}^{m_2}}(f_k)\|^2\\
&=\frac{1}{n_1+k+m_1+2}-\frac{2k}{(k+1)(n_1+k+m_1+1)}\\
&\quad-\frac{2n_1(n_1+k-m_1)}{(k+1)(n_1+k)(n_1+k+1)^2}+\frac{k^2}{(k+1)^2(n_1+k+m_1)}\\
&\quad+\frac{2kn_1(n_1+k-m_1)}{(k+1)^2(n_1+k)^2(n_1+k+1)}+\frac{n_{1}^2(n_1+k-m_1)}{(k+1)^2(n_1+k)^2(n_1+k+1)^2}\\
&\quad+\frac{|\alpha|^2}{n_2+k+m_2+2}-\frac{2|\alpha|^2k}{(k+1)(n_2+k+m_2+1)}\\
&\quad-\frac{2|\alpha|^2n_2(n_2+k-m_2)}{(k+1)(n_2+k)(n_2+k+1)^2}+\frac{|\alpha|^2k^2}{(k+1)^2(n_2+k+m_2)}\\
&\quad+\frac{2|\alpha|^2kn_2(n_2+k-m_2)}{(k+1)^2(n_2+k)^2(n_2+k+1)}+\frac{|\alpha|^2n_{2}^2(n_2+k-m_2)}{(k+1)^2(n_2+k)^2(n_2+k+1)^2}.
\end{align*}
Similarly, we also have
\begin{align*}
\|S_{\varphi}^*f_k\|^2&=\|S_{z^{n_1}\overline{z}^{m_1}}^{*}(f_k)\|^2+|\alpha|^2\|S_{z^{n_2}\overline{z}^{m_2}}^{*}(f_k)\|^2\\
&=\frac{1}{m_1+k+n_1+2}-\frac{2k}{(k+1)(m_1+k+n_1+1)}\\
&\quad-\frac{2m_1(m_1+k-n_1)}{(k+1)(m_1+k)(m_1+k+1)^2}+\frac{k^2}{(k+1)^2(m_1+k+n_1)}\\
&\quad+\frac{2km_1(m_1+k-n_1)}{(k+1)^2(m_1+k)^2(m_1+k+1)}+\frac{m_{1}^2(m_1+k-n_1)}{(k+1)^2(m_1+k)^2(m_1+k+1)^2}\\
&\quad+\frac{|\alpha|^2}{m_2+k+n_2+2}-\frac{2|\alpha|^2k}{(k+1)(m_2+k+n_2+1)}\\
&\quad-\frac{2|\alpha|^2m_2(m_2+k-n_2)}{(k+1)(m_2+k)(m_2+k+1)^2}+\frac{|\alpha|^2k^2}{(k+1)^2(m_2+k+n_2)}\\
&\quad+\frac{2|\alpha|^2km_2(m_2+k-n_2)}{(k+1)^2(m_2+k)^2(m_2+k+1)}+\frac{|\alpha|^2m_{2}^2(m_2+k-n_2)}{(k+1)^2(m_2+k)^2(m_2+k+1)^2}.
\end{align*}
Combining the above two equalities gives
\begin{equation}\label{eq8}
\begin{aligned}
0&=\big\langle (S_{\varphi}^*S_{\varphi}-S_{\varphi}S_{\varphi}^*)(f_k),f_k\big\rangle=\|S_{\varphi}f_k\|^2-\|S_{\varphi}^*f_k\|^2\\
&=\|S_{z^{n_1}\overline{z}^{m_1}}(f_k)\|^2+|\alpha|^2\|S_{z^{n_2}\overline{z}^{m_2}}(f_k)\|^2
-\|S_{z^{n_1}\overline{z}^{m_1}}^{*}(f_k)\|^2-|\alpha|^2\|S_{z^{n_2}\overline{z}^{m_2}}^{*}(f_k)\|^2\\
&=\frac{-n_{1}^2(n_1+k-m_1)}{(k+1)^2(n_1+k)^2(n_1+k+1)^2}+\frac{m_{1}^2(m_1+k-n_1)}{(k+1)^2(m_1+k)^2(m_1+k+1)^2}\\
&\quad+\frac{-|\alpha|^2n_{2}^2(n_2+k-m_2)}{(k+1)^2(n_2+k)^2(n_2+k+1)^2}+\frac{|\alpha|^2m_{2}^2(m_2+k-n_2)}{(k+1)^2(m_2+k)^2(m_2+k+1)^2}.
\end{aligned}
\end{equation}
Therefore, we conclude  that
\begin{equation}\label{eq5}
\begin{aligned}
&\quad\frac{n_{1}^2(n_1+k-m_1)}{(n_1+k)^2(n_1+k+1)^2}+\frac{|\alpha|^2n_{2}^2(n_2+k-m_2)}{(n_2+k)^2(n_2+k+1)^2}\\
&=\frac{m_{1}^2(m_1+k-n_1)}{(m_1+k)^2(m_1+k+1)^2}+\frac{|\alpha|^2m_{2}^2(m_2+k-n_2)}{(m_2+k)^2(m_2+k+1)^2}
\end{aligned}
\end{equation}
for any integer $k>\max\{n_1,n_2,m_1,m_2\}$.

 In order to deduce the desired conclusion, we will use the same method as the one in the proof of Proposition \ref{2.c}. Multiplying both sides of (\ref{eq5}) by
$$\left[(n_1+k)(n_1+k+1)(n_2+k)(n_2+k+1)(m_1+k)(m_1+k+1)(m_2+k)(m_2+k+1)\right]^2,$$
now we obtain a polynomial $p(z)$ with infinitely many zeros. Thus $p$ must be a zero polynomial  and (\ref{eq5}) holds for
all $k\in \N$. By substituting $k=0$ into (\ref{eq5}), we get
\begin{align}\label{mn}
\frac{n_1-m_1}{(n_1+1)^2}+\frac{|\alpha|^2(n_2-m_2)}{(n_2+1)^2}
=-\frac{n_1-m_1}{(m_1+1)^2}-\frac{|\alpha|^2(n_2-m_2)}{(m_2+1)^2}.
\end{align}

If $n_1=m_1$ or $n_2=m_2$, then we have
$$n_1-m_1=n_2-m_2=0,$$
since $\alpha\neq 0$. This  is a contradiction. Otherwise, we have by our assumption that
$$(n_1-m_1)(n_2-m_2)>0.$$
In this case, the left-hand side and the right-hand side of (\ref{mn}) have  opposite signs. This is also impossible. Thus we finish the proof of this lemma.
\end{proof}

Under the assumption in Lemma \ref{2.d}, we can further show that $n_1=m_1$ and $n_2=m_2$.

\begin{lem}\label{2.g}
Let $n_1,n_2,m_1,m_2$ be positive integers such that
$$(n_1-m_1)(n_2-m_2)\geqslant 0.$$
Suppose that $\varphi(z)=z^{n_1}\overline{z}^{m_1}+\alpha z^{n_2}\overline{z}^{m_2}$ is a nonzero function, where $\alpha$ is a nonzero  constant. Then  $n_1=m_1$ and $n_2=m_2$ if $S_{\varphi}$ is normal on $(L_h^2)^\perp$.
\end{lem}
\begin{proof}
Recall that we have shown in the above lemma that
$$n_1-m_1=n_2-m_2.$$
Using the same notation as the one in Lemma \ref{2.d}, we will  calculate the norms $\|S_\varphi f_k\|$ and $\|S_\varphi^* f_k\|$. To do this, the following equalities are needed (which can be calculated  by using (\ref{1.a})):
\begin{align*}
\langle z^{n_1+k}\overline{z}^{m_1+1},z^{n_2+k}\overline{z}^{m_2+1}\rangle&=\langle z^{m_1+k}\overline{z}^{n_1+1},z^{m_2+k}\overline{z}^{n_2+1}\rangle\\
&=\frac{2}{n_1+m_1+n_2+m_2+2k+4},
\end{align*}
\begin{align*}
\langle z^{n_1+k}\overline{z}^{m_1+1},z^{n_2+k-1}\overline{z}^{m_2}\rangle&=\langle z^{m_1+k}\overline{z}^{n_1+1},z^{m_2+k-1}\overline{z}^{n_2}\rangle\\
&=\langle z^{n_1+k-1}\overline{z}^{m_1},z^{n_2+k}\overline{z}^{m_2+1}\rangle\\
&=\langle z^{m_1+k-1}\overline{z}^{n_1},z^{m_2+k}\overline{z}^{n_2+1}\rangle\\
&=\frac{2}{n_1+m_1+n_2+m_2+2k+2},
\end{align*}
\begin{align*}
\langle z^{n_1+k-1}\overline{z}^{m_1},z^{n_2+k-1}\overline{z}^{m_2}\rangle&=\langle z^{m_1+k-1}\overline{z}^{n_1},z^{m_2+k-1}\overline{z}^{n_2}\rangle\\
&=\frac{2}{n_1+m_1+n_2+m_2+2k},
\end{align*}

$$\langle z^{n_1+k}\overline{z}^{m_1+1},z^{n_2+k-m_2-1}\rangle=\frac{1}{n_1+k+1}, \ \ \langle z^{n_1+k-1}\overline{z}^{m_1},z^{n_2+k-m_2-1}\rangle=\frac{1}{n_1+k},$$
$$\langle z^{n_1+k-m_1-1},z^{n_2+k}\overline{z}^{m_2+1}\rangle=\frac{1}{n_2+k+1},\ \ \langle z^{n_1+k-m_1-1},z^{n_2+k-1}\overline{z}^{m_2}\rangle=\frac{1}{n_2+k},$$
$$\langle z^{m_1+k}\overline{z}^{n_1+1},z^{m_2+k-n_2-1}\rangle=\frac{1}{m_1+k+1},  \ \ \langle z^{m_1+k-1}\overline{z}^{n_1},z^{m_2+k-n_2-1}\rangle=\frac{1}{m_1+k}, $$
$$\langle z^{m_1+k-n_1-1},z^{m_2+k}\overline{z}^{n_2+1}\rangle=\frac{1}{m_2+k+1},\ \ \langle z^{m_1+k-n_1-1},z^{m_2+k-1}\overline{z}^{n_2}\rangle=\frac{1}{m_2+k}. $$
Since $n_1-m_1=n_2-m_2$, we also have
$$\langle z^{n_1+k-m_1-1},z^{n_2+k-m_2-1}\rangle=\frac{1}{n_1-m_1+k}=\frac{1}{n_2-m_2+k}$$
and
$$\langle z^{m_1+k-n_1-1},z^{m_2+k-n_2-1}\rangle=\frac{1}{m_1-n_1+k}=\frac{1}{m_2-n_2+k}.$$

From the  expressions of $S_{z^{n_1}\overline{z}^{m_1}}(f_k)$, $S_{z^{n_2}\overline{z}^{m_2}}(f_k)$, $S_{z^{n_1}\overline{z}^{m_1}}^*(f_k)$ and $S_{z^{n_2}\overline{z}^{m_2}}^*(f_k)$ obtained in  Lemma \ref{2.d}, we observe that $\langle S_{z^{n_1}\overline{z}^{m_1}}(f_k),S_{z^{n_2}\overline{z}^{m_2}}(f_k)\rangle$ and $\langle S_{z^{n_1}\overline{z}^{m_1}}^*(f_k),S_{z^{n_2}\overline{z}^{m_2}}^*(f_k)\rangle$ are both real.
Thus we have
\begin{align*}
\|S_{\varphi}f_k\|^2&=
\big\langle S_{z^{n_1}\overline{z}^{m_1}}(f_k)+\alpha S_{z^{n_2}\overline{z}^{m_2}}(f_k),S_{z^{n_1}\overline{z}^{m_1}}(f_k)+\alpha S_{z^{n_2}\overline{z}^{m_2}}(f_k)\big\rangle\\
&=\|S_{z^{n_1}\overline{z}^{m_1}}(f_k)\|^2+(\overline{\alpha}+\alpha)\langle S_{z^{n_1}\overline{z}^{m_1}}(f_k),S_{z^{n_2}\overline{z}^{m_2}}(f_k)\rangle+|\alpha|^2\|S_{z^{n_2}\overline{z}^{m_2}}(f_k)\|^2
\end{align*}
and
\begin{align*}
\|S_{\varphi}^*f_k\|^2&=
\big\langle S_{z^{n_1}\overline{z}^{m_1}}^*(f_k)+\overline{\alpha}S_{z^{n_2}\overline{z}^{m_2}}^*(f_k),S_{z^{n_1}\overline{z}^{m_1}}^*(f_k)+\overline{\alpha}
S_{z^{n_2}\overline{z}^{m_2}}^*(f_k)\big\rangle\\
&=\|S_{z^{n_1}\overline{z}^{m_1}}^*(f_k)\|^2+(\alpha+\overline{\alpha})\langle S_{z^{n_1}\overline{z}^{m_1}}^*(f_k),S_{z^{n_2}\overline{z}^{m_2}}^*(f_k)\rangle+|\alpha|^2\|S_{z^{n_2}\overline{z}^{m_2}}^*(f_k)\|^2.
\end{align*}
It follows that
\begin{equation}\label{eq9}
\begin{aligned}
0&=\|S_{\varphi}f_k\|^2-\|S_{\varphi}^*f_k\|^2\\
&=\|S_{z^{n_1}\overline{z}^{m_1}}(f_k)\|^2+|\alpha|^2\|S_{z^{n_2}\overline{z}^{m_2}}(f_k)\|^2-\|S_{z^{n_1}\overline{z}^{m_1}}^*(f_k)\|^2
-|\alpha|^2\|S_{z^{n_2}\overline{z}^{m_2}}^*(f_k)\|^2\\
&\quad+(\alpha+\overline{\alpha})\big[\langle S_{z^{n_1}\overline{z}^{m_1}}(f_k),S_{z^{n_2}\overline{z}^{m_2}}(f_k)\rangle-\langle S_{z^{n_1}\overline{z}^{m_1}}^*(f_k),S_{z^{n_2}\overline{z}^{m_2}}^*(f_k)\rangle\big]\\
&=\frac{-n_{1}^2(n_1+k-m_1)}{(k+1)^2(n_1+k)^2(n_1+k+1)^2}+\frac{m_{1}^2(m_1+k-n_1)}{(k+1)^2(m_1+k)^2(m_1+k+1)^2}\\
&\quad+\frac{-|\alpha|^2n_{2}^2(n_2+k-m_2)}{(k+1)^2(n_2+k)^2(n_2+k+1)^2}+\frac{|\alpha|^2m_{2}^2(m_2+k-n_2)}{(k+1)^2(m_2+k)^2(m_2+k+1)^2}\\
&\quad+\frac{-(\alpha+\overline{\alpha})n_1n_2(n_2+k-m_2)}{(k+1)^2(n_1+k)(n_1+k+1)(n_2+k)(n_2+k+1)}\\
&\quad+\frac{(\alpha+\overline{\alpha})m_1m_2(m_2+k-n_2)}{(k+1)^2(m_1+k)(m_1+k+1)(m_2+k)(m_2+k+1)},
\end{aligned}
\end{equation}
where the third equality comes from (\ref{eq8}). Note that the above equality holds for all $k>\max\{n_1,n_2,m_1,m_2\}$. Using the same techniques as in the proof of Lemma \ref{2.d}, we see that (\ref{eq9}) is, in fact, an identity. Thus (\ref{eq9}) can be  reduced to
$$(m_1-n_1)\left[\left|\frac{1}{n_1+1}+\frac{\alpha}{n_2+1}\right|^2+\left|\frac{1}{m_1+1}+\frac{\alpha}{m_2+1}\right|^2\right]=0$$
if we take $k=0$.   This implies  that $n_1=m_1$ or
$$\left|\frac{1}{n_1+1}+\frac{\alpha}{n_2+1}\right|^2+\left|\frac{1}{m_1+1}+\frac{\alpha}{m_2+1}\right|^2=0.$$

Suppose that $n_1\neq m_1$, then we would have
$$-\alpha=\frac{n_2+1}{n_1+1}=\frac{m_2+1}{m_1+1},$$
to obtain $$n_2+m_1+n_2m_1=n_1+m_2+n_1m_2.$$ Since $n_1-m_1=n_2-m_2$, we have $n_2m_1=n_1m_2$. Thus
$$(n_1+m_2)^2=n_1^2+m_2^2+2n_1m_2=(n_2+m_1)^2=n_2^2+m_1^2+2n_2m_1,$$
which gives that $n_1^2-m_1^2=n_2^2-m_2^2$. As $n_1-m_1=n_2-m_2\neq 0$, we have  $n_1+m_1=n_2+m_2$. But this implies that
$n_1=n_2$ and $m_1=m_2$. In this case, $\varphi$  can be rewritten as $$\varphi(z)=(1+\alpha)z^{n_1}\overline{z}^{m_1}.$$  By our assumption that $\varphi$ is a nonzero function,  we now  conclude by Proposition \ref{2.c} that $n_1=m_1$.  This is a contradiction. Thus we obtain that $n_1=m_1$ and $n_2=m_2$, to complete the proof of Lemma \ref{2.g}.
\end{proof}

For functions of the form $\varphi(z)=z^{n_1}\overline{z}^{m_1}+\alpha z^{n_2}\overline{z}^{m_2}$, we will discuss in the next lemma for the case that
$(n_1-m_1)(n_2-m_2)<0$.

\begin{lem}\label{2.h}
Let $\varphi(z)=z^{n_1}\overline{z}^{m_1}+\alpha z^{n_2}\overline{z}^{m_2}$ be a nonzero function, where $\alpha$ is a nonzero complex constant and  $n_1,n_2,m_1,m_2$ are positive integers. Suppose that
$$(n_1-m_1)(n_2-m_2)<0.$$
Then we have  $|\alpha|=1$, $n_1=m_2$ and $n_2=m_1$ if  $S_{\varphi}$ is normal on $(L_h^2)^{\perp}$.
\end{lem}
\begin{proof}
Without loss of generality, we may assume that
$$n_1>m_1\ \ \ \ \text{and}\ \ \ \  n_2<m_2.$$
Since $n_1-m_1\neq n_2-m_2$, we observe that (\ref{eq5}) in the proof of Lemma \ref{2.d} is still  valid in this case, i.e.,
\begin{align*}\label{eq6}
&\quad\frac{n_{1}^2(n_1+k-m_1)}{(n_1+k)^2(n_1+k+1)^2}+\frac{|\alpha|^2n_{2}^2(n_2+k-m_2)}{(n_2+k)^2(n_2+k+1)^2}\\
&=\frac{m_{1}^2(m_1+k-n_1)}{(m_1+k)^2(m_1+k+1)^2}+\frac{|\alpha|^2m_{2}^2(m_2+k-n_2)}{(m_2+k)^2(m_2+k+1)^2}
\end{align*}
for any integer $k>\max\{n_1,n_2,m_1,m_2\}$.
For convenience, we write
$$a(k)=\frac{n_{1}^2(n_1+k-m_1)}{(n_1+k)^2(n_1+k+1)^2},\ \ b(k)=\frac{n_{2}^2(n_2+k-m_2)}{(n_2+k)^2(n_2+k+1)^2},$$
$$c(k)=\frac{m_{1}^2(m_1+k-n_1)}{(m_1+k)^2(m_1+k+1)^2},\ \ d(k)=\frac{m_{2}^2(m_2+k-n_2)}{(m_2+k)^2(m_2+k+1)^2}.$$

Let
$$F(k)=a(k)+|\alpha|^2b(k)-c(k)-|\alpha|^2d(k)$$
and
$$G(k)=\big[(n_1+k)(n_1+k+1)(n_2+k)(n_2+k+1)(m_1+k)(m_1+k+1)(m_2+k)(m_2+k+1)\big]^2.$$
Then $F(k)=0$ for all  $k>\max\{n_1, n_2, m_1, m_2\}$.
Let
$$H(x)=F(x)G(x).$$  Note that $H(x)$ is a polynomial  of $x$ and $H(x)$ has infinitely many zeros.  This yields  that $H(x)$ is a zero polynomial.
Considering the coefficient of $x^{13}$ in $H(x)$, we have that
\begin{equation}\label{eq7}
n_{1}^2+|\alpha|^2n_{2}^2=m_{1}^2+|\alpha|^2m_{2}^2.
\end{equation}

Now we are going to show that $n_1=m_2$. If not, we first assume that $n_1>m_2$. Recalling that $n_1>m_1$ and $n_2<m_2$, we obtian
$$n_1>\max\{n_2,m_1,m_2\}.$$
This gives that
\begin{align*}
b(-n_{1}-1)G(-n_{1}-1)&=c(-n_{1}-1)G(-n_{1}-1)\\
&=d(-n_{1}-1)G(-n_{1}-1)\\
&=0.
\end{align*}
It follows that
$$F(-n_{1}-1)G(-n_{1}-1)=a(-n_{1}-1)G(-n_{1}-1)=0.$$
On the other hand,
\begin{align*}
a(-n_{1}-1)G(-n_{1}-1)&=-n_{1}^2(1+m_1)\bigg(\frac{g(x)}{(n_{1}+x)^2(n_1+1+x)^2}\bigg)\bigg|_{x=-n_{1}-1}\\
&\neq 0.
\end{align*}
The contradiction gives  $n_1\leqslant m_2$. Similarly, we can show that $n_1\geqslant m_2$. So we have  $n_1=m_2$.

To finish the proof, it remains to show that $m_1=n_2$.  If $m_1>n_2$, then
$$n_2<\min\{n_1,m_1,m_2\}.$$
This  implies that
$$a(-n_{2})G(-n_{2})=c(-n_{2})G(-n_{2})=d(-n_{2})G(-n_{2})=0,$$
to obtain
$$F(-n_{2})G(-n_{2})=|\alpha|^2b(-n_{2})G(-n_{2})=0.$$
However,
$$b(-n_{2})G(-n_{2})=-m_2n_{2}^2\bigg(\frac{G(x)}{(n_{2}+x)^2(n_2+1+x)^2}\bigg)\bigg|_{x=-n_{2}}\neq 0,$$
which is a contradiction, so $m_1\leqslant  n_2$. Using the same argument as above, we can also show that $m_1\geqslant n_2$. Thus we get $n_2=m_1$. From (\ref{eq7}) and $(n_1-m_1)(n_2-m_2)<0$, we conclude that
$|\alpha|=1$. This completes the proof of Lemma \ref{2.h}.
\end{proof}

Before presenting the proof of Theorem \ref{M2}, one more lemma is needed.
\begin{lem}\label{2.i}
Suppose that $\varphi=|z|^{2n}+\alpha|z|^{2m}$, where $\alpha$ is a constant, $m$ and $n$ are distinct integers. Then
$S_{\varphi}$ is normal if and only if $\alpha$ is real.
\end{lem}
\begin{proof}
If $\alpha\in \mathbb R$, then $S_{\varphi}$ is self-adjoint, and so is normal. Now we assume that $S_{\varphi}$ is normal. It follows that
\begin{align*}
0&=S_{\varphi}^*S_{\varphi}-S_{\varphi}S_{\varphi}^*\\
&=(S_{|z|^{2n}}+\overline{\alpha}S_{|z|^{2m}})(S_{|z|^{2n}}+\alpha S_{|z|^{2m}})-(S_{|z|^{2n}}+\alpha S_{|z|^{2m}})(S_{|z|^{2n}}+\overline{\alpha}S_{|z|^{2m}})\\
&=(\alpha-\overline{\alpha})(S_{|z|^{2n}}S_{|z|^{2m}}-S_{|z|^{2m}}S_{|z|^{2n}}).
\end{align*} For
$$f_0(z)=(I-Q)(|z|^2)=\Big(|z|^2-\frac{1}{2}\Big)\in (L_h^2)^{\bot},$$
we have
\begin{align*}
0&=(\alpha-\overline{\alpha})(S_{|z|^{2n}}S_{|z|^{2m}}-S_{|z|^{2m}}S_{|z|^{2n}})(f_0)\\
&=(\alpha-\overline{\alpha})\Big\{(I-Q)\left[|z|^{2n}(I-Q)(|z|^{2m}f_0)\right]-(I-Q)\left[|z|^{2m}(I-Q)(|z|^{2n}f_0)\right]\Big\}\\
&=(\alpha-\overline{\alpha})(I-Q)\left[|z|^{2m}Q(|z|^{2n}f_0)-|z|^{2n}Q(|z|^{2m}f_0)\right].
\end{align*}
On the one hand,
\begin{align*}
|z|^{2m}Q(|z|^{2n}f_0)-|z|^{2n}Q(|z|^{2m}f_0)&= |z|^{2m}Q(|z|^{2(n+1)}-\frac{1}{2}|z|^{2n})-|z|^{2n}Q(|z|^{2(m+1)}-\frac{1}{2}|z|^{2m})\\
&=|z|^{2m}\left(\frac{1}{n+2}-\frac{1}{2(n+1)}\right)-|z|^{2n}\left(\frac{1}{m+2}-\frac{1}{2(m+1)}\right)\\
&=\frac{n|z|^{2m}}{2(n+1)(n+2)}-\frac{m|z|^{2n}}{2(m+1)(m+2)}.
\end{align*}
On the other hand, we have
\begin{align*}
Q\left[|z|^{2m}Q(|z|^{2n}f_0)-|z|^{2n}Q(|z|^{2m}f_0)\right]&=Q\left[\frac{n|z|^{2m}}{2(n+1)(n+2)}-\frac{m|z|^{2n}}{2(m+1)(m+2)}\right]\\
&=\frac{n}{2(n+1)(n+2)(m+1)}-\frac{m}{2(m+1)(m+2)(n+1)}\\
&=\frac{n-m}{(n+1)(n+2)(m+1)(m+2)}.
\end{align*}
Since $m\neq n$, we derive  that $$(I-Q)\big[|z|^{2m}Q(|z|^{2n}f_0)-|z|^{2n}Q(|z|^{2m}f_0)\big]\neq 0,$$
to obtain $\alpha=\overline{\alpha}$. This completes the proof of Lemma \ref{2.i}.
\end{proof}

We are now in  position to prove Theorem \ref{M2}. Let $\varphi(z)=az^{n_1}\overline{z}^{m_1}+bz^{n_2}\overline{z}^{m_2}$. Recall that we need to show that $S_{\varphi}$ is is normal on $(L_h^2)^{\perp}$ if and only if one of the following holds:\\
$\mathrm{(1)}$ $n_1=m_1=n_2=m_2$;\\
$\mathrm{(2)}$ $n_1=m_1, n_2=m_2, n_1\neq n_2$ and $\arg(a)=\arg(b)$;\\
$\mathrm{(3)}$ $n_1=m_2, n_2=m_1, n_1\neq m_1$ and $|a|=|b|$.
\begin{proof}[\emph{\textbf{\emph{Proof of  Theorem \ref{M2}.}}}]
Suppose that $(1)$ or $(2)$ holds.  We conclude that $S_{\varphi}$ is normal by Proposition \ref{2.c} and Lemma \ref{2.i},
respectively. If $(3)$ holds, then there exists $\theta\in [0,2\pi]$ such that $b=ae^{i\theta}$. In this case, we have
\begin{align*}
S_{\varphi}^*S_{\varphi}-S_{\varphi}S_{\varphi}^*&=|a|^2(S_{\overline{z}^{n_1}z^{m_1}}+e^{-i\theta}S_{z^{n_1}\overline{z}^{m_1}})
(S_{z^{n_1}\overline{z}^{m_1}}+e^{i\theta}S_{\overline{z}^{n_1}z^{m_1}})\\
&\quad-|a|^2(S_{z^{n_1}\overline{z}^{m_1}}+e^{i\theta}S_{\overline{z}^{n_1}z^{m_1}})
(S_{\overline{z}^{n_1}z^{m_1}}+e^{-i\theta}S_{z^{n_1}\overline{z}^{m_1}})\\
&=|a|^2e^{-i\theta}(e^{i\theta}S_{\overline{z}^{n_1}z^{m_1}}+S_{z^{n_1}\overline{z}^{m_1}})
(e^{-i\theta}S_{z^{n_1}\overline{z}^{m_1}}+S_{\overline{z}^{n_1}z^{m_1}})e^{i\theta}\\
&\quad-|a|^2(S_{z^{n_1}\overline{z}^{m_1}}+e^{i\theta}S_{\overline{z}^{n_1}z^{m_1}})
(S_{\overline{z}^{n_1}z^{m_1}}+e^{-i\theta}S_{z^{n_1}\overline{z}^{m_1}})\\
&=0,
\end{align*}
 which implies that $S_{\varphi}$ is a normal operator.

The necessary part of the theorem follows from Lemmas \ref{2.g}, \ref{2.h} and \ref{2.i} immediately. This completes the proof of Theorem \ref{M2}.
\end{proof}

\section{Finite rank commutator}\label{4}
In the last section, we will study the rank of the commutator of two dual Toeplitz operators on $(L_h^2)^\perp$. For $f,g\in (L_h^2)^{\perp}$, recall that the rank-one operator $f\otimes g$ on $(L_h^2)^{\perp}$ is defined by
$$(f\otimes g)(h)=\langle h,g\rangle f,\ \ \ \  h\in (L_h^2)^{\perp}.$$
Our main result of this section is the following theorem, which is analogous to the characterization for finite rank commutator of Toeplitz operators on the harmonic Bergman space \cite[Theorem 2.2]{CKL}.
\begin{thm}\label{M3}
Let $\varphi$ and $\psi$ be in $L^{\infty}(\D)$. If the commutator
$$[S_{\varphi},S_{\psi}]=S_{\varphi}S_{\psi}-S_{\psi}S_{\varphi}$$
has a finite rank, then the rank  of $[S_{\varphi},S_{\psi}]$ is an even number.
\end{thm}
\begin{proof}
Suppose that the rank of $[S_{\varphi},S_{\psi}]$ is $n$. Then we can find two linearly independent sequences $\left\{s_i\right\}_{i=1}^{n}$ and
$\left\{t_i\right\}_{i=1}^{n}$ in $\left(L_h^2\right)^{\perp}$, which satisfy that
$$[S_{\varphi},S_{\psi}]=\sum_{i=1}^{n}s_i\otimes t_i.$$
Define an anti-unitary operator $U$ on $\left(L_h^2\right)^{\perp}$ as the following:
$$Uf=\overline{f},    \ \ \  f\in\left(L_h^2\right)^{\perp}.$$
Noting that $(I-Q)(\overline{f})=\overline{(I-Q)(f)}$ for any $f\in L^2(\D)$,  we have
\begin{align*}
(S_{\varphi}S_{\psi}-S_{\psi}S_{\varphi})^*(h)&= (S_{\overline{\psi}}S_{\overline{\varphi}}-S_{\overline{\varphi}}S_{\overline{\psi}})(h)\\
&= \overline{(S_{\psi}S_{\varphi}-S_{\varphi}S_{\psi})(\overline{h})}\\
&= U(S_{\psi}S_{\varphi}-S_{\varphi}S_{\psi})U(h)\\
&= -U[S_{\varphi},S_{\psi}]U(h)
\end{align*}
for each $h\in \left(L_h^2\right)^{\perp}$.
This gives that
$$\left(\sum_{i=1}^{n}s_i\otimes t_i\right)^*=-U\left(\sum_{i=1}^{n}s_i\otimes t_i\right)U,$$
which is equivalent to
$$\sum_{i=1}^{n}t_i\otimes s_i=-\sum_{i=1}^{n}\overline{s_i}\otimes \overline{t_i}.$$
Thus we can find $\{a_{ij}\}_{j=1}^n\subset \C$ such that
$$t_i=\sum_{j=1}^{n}a_{ij}\overline{s_j}$$
for each $t_i$. Then we have
\begin{align*}
\sum_{i=1}^{n}t_i\otimes s_i&=\sum_{i=1}^{n}(\sum_{j=1}^{n}a_{ij}\overline{s_j})\otimes s_i\\
&= \sum_{i=1}^{n}\sum_{j=1}^{n}a_{ij}\overline{s_j}\otimes s_i\\
&= \sum_{j=1}\overline{s_j}\otimes \sum_{i=1}^{n}\overline{a_{ij}}s_i\\
&= -\sum_{i=1}^{n}\overline{s_i}\otimes \overline{t_i}.
\end{align*}
It follows that
$$\overline{t_i}=-\sum_{i=1}^{n}\overline{a_{ji}}s_j,$$
to obtain $$t_i=-\sum\limits_{i=1}^{n}a_{ji}\overline{s_j}=\sum\limits_{j=1}^{n}a_{ij}\overline{s_j}.$$

Since $\left\{s_1, s_2, \cdots, s_n\right\}$ is
linearly independent, we conclude that
$$a_{ij}=-a_{ji}.$$
Denote the matrix $(a_{ij})_{n\times n}$ by $A$, then $\det(A)\neq 0$. Hence $A=-A^T$ and
$$\det(A)=(-1)^n\cdot\det(A^T)=(-1)^n\cdot\det(A),$$
where $A^T$ is the transpose of the matrix $A$. This implies that $(-1)^n=1$ as $\det(A)\neq 0$, so $n$ must be even. This completes the proof of Theorem \ref{M3}.
\end{proof}\vspace{5mm}

\subsection*{Acknowledgment}
This work was partially supported by  NSFC (grant number: 11701052). The second author was partially supported by the Fundamental Research Funds for the Central Universities (grant numbers: 2020CDJQY-A039, 2020CDJ-LHSS-003).

\end{document}